\documentclass[11pt]{amsart}
\usepackage{amsfonts}
\usepackage{amsmath,amssymb,amsthm}
\usepackage{color,graphics}
\usepackage{bm}
\usepackage{amsmath}
\usepackage{amssymb}
\usepackage{graphicx}
\usepackage{amsmath,amsfonts,amssymb,amsthm,graphics,amscd,enumerate,verbatim}
\usepackage{color}
\usepackage[T1]{fontenc}%
\newtheorem{theorem}{Theorem}
\newtheorem{corollary}{Corollary}[section]
\newtheorem{proposition}{Proposition}[section]
\newtheorem{lemma}[proposition]{Lemma}

\theoremstyle{definition}
\newtheorem{remark}{Remark}[section]
\theoremstyle{plain}
\begin{document}
\title[Preservers on the set of variance-covariance matrices]{Preservers of partial orders on the set of all variance-covariance matrices}
\author{Iva Golubi\'{c}}
\address[Iva Golubi\'{c}]{University of Applied Sciences Velika Gorica, Zagreba\v{c}ka cesta 5,
10410 Velika Gorica, Croatia}
\email{iva.golubic@gmail.com}
\author{Janko Marovt}
\address[Janko Marovt]{Faculty of Economics and Business, University of Maribor,
Razlagova 14, SI-2000 Maribor, Slovenia}
\email{janko.marovt@um.si}

\begin{abstract}
Let $H_{n}^{+}(\mathbb{R})$ be the cone of all positive semidefinite
$n\times n$ real matrices. Two of the best known partial orders that were
mostly studied on subsets of square complex matrices are the L{\"{o}}wner and
the minus partial orders. Motivated by applications in statistics we study
these partial orders on $H_{n}^{+}(\mathbb{R})$. We describe the form of all
surjective maps on $H_{n}^{+}(\mathbb{R})$, $n>1$, that preserve the
L{\"{o}}wner partial order in both directions. We present an equivalent
definition of the minus partial order on $H_{n}^{+}(\mathbb{R})$ and also
characterize all surjective, additive maps on $H_{n}^{+}(\mathbb{R})$,
$n\geq3$, that preserve the minus partial order in both directions.

\end{abstract}
\maketitle

\section{Introduction}

Let $M_{m,n}(\mathbb{F})$ where $\mathbb{F=R}$ or $\mathbb{F=C}$ be the set of
all $m\times n$ real or complex matrices, let $A^{t}\in M_{n,m}(\mathbb{F})$
denote the transpose, $A^{\ast}\in M_{n,m}(\mathbb{F})$ the conjugate
transpose, $\operatorname{Im}A$ the image (i.e. the column space), and
$\operatorname*{Ker}A$ the kernel (the nullspace) of $A\in M_{m,n}%
(\mathbb{F})$. Any matrix which is a solution $X=A^{-}\in M_{n,m}(\mathbb{F})$
to the equation $AXA=A$ is called an inner generalized inverse of $A\in
M_{m,n}(\mathbb{F})$. Note that every matrix $A\in M_{m,n}(\mathbb{F})$ has an
inner generalized inverse (see e.g. \cite{MitraKnjiga}). If $m=n$, then we
will write $M_{n}(\mathbb{F)}$ instead of $M_{n,n}(\mathbb{F)}$. We say that
$A\in M_{n}(\mathbb{F)}$ is symmetric if $A=A^{t}$ and Hermitian (or
self-adjoined) if $A=A^{\ast}$. A symmetric matrix $A\in M_{n}(\mathbb{R)}$ is
said to be positive semidefinite if $x^{t}Ax\geq0$ for every $x\in
\mathbb{R}^{n}$. More generally, a Hermitian matrix $A\in$ $M_{n}(\mathbb{C)}$
is said to be positive semidefinite if $z^{\ast}Az\geq0$ for every
$z\in\mathbb{C}^{n}$. The study of positive semidefinite matrices is a
flourishing area of mathematical investigation (see e.g., the monograph
\cite{Bahtia} and the references therein). Moreover, positive semidefinite
matrices have become fundamental computational objects in many areas of
statistics, engineering, quantum information, and applied mathematics. They
appear as variance-covariance matrices (also known as dispersion or covariance
matrices) in statistics, as elements of the search space in convex and
semidefinite programming, as kernels in machine learning, as density matrices
in quantum information, and as diffusion tensors in medical imaging. It is
known (see e.g. \cite{Christensen}) that every variance-covariance matrix is
positive semidefinite, and that every (real) positive semidefinite matrix is a
variance-covariance matrix of some multivariate distribution.

There are many partial orders which may be defined on various sets of
matrices. We will next present two of the best known. Let $A,B\in
M_{n}(\mathbb{R)}$ be symmetric matrices. Then we say that $A$ is below $B$
with respect to the L{\"{o}}wner partial order and write
\begin{equation}
A\leq^{L}B\quad\text{if}\quad B-A\quad\text{is positive semidefinite.}
\label{Def_Lowner}%
\end{equation}
L{\"{o}}wner partial order has many applications in statistics especially in
the theory of linear statistical models. Let
\[
y=X\beta+\epsilon
\]
be the matrix form of a linear model. Here $y$ is a real $n\times1$ random
vector of observed quantities which we try to explain with other quantities
that determine the matrix $X\in$ $M_{n,p}(\mathbb{R})$. It is assumed that
$E(\epsilon)=0$ and $V(\epsilon)=\sigma^{2}D$, i.e. the errors have the zero
mean and covariances are known up to a scalar (real number). Here $V$ denotes
the variance-covariance matrix. The nonnegative parameter $\sigma^{2}$ and the
vector of parameters (real numbers) $\beta$ are unspecified, and $D\in
M_{n}(\mathbb{R)}$ is a known positive semidefinite matrix. We denote this
linear model with the triplet $(y,X\beta,\sigma^{2}D)$.

Classical inference problems related to the linear model $(y,X\beta,\sigma
^{2}D)$ usually concern a vector linear parametric function (LPF), $A\beta$
(here $A$ is a real matrix with $p$ columns). We try to estimate it by a
linear function of the response $Cy$ (here $C$ is a real matrix with $n$
columns).
We say that the statistic $Cy$ is a linear unbiased estimator (LUE) of
$A\beta$ if $E(Cy)=A\beta$ for all possible values of $\beta\in\mathbb{R}^{p}%
$. A vector LPF is said to be estimable if it has an LUE. The best linear
unbiased estimator (BLUE) of an estimable vector LPF is defined as the LUE
having the smallest variance-covariance matrix. Here, the \textquotedblleft
variance-covariance\textquotedblright\ condition is expressed in terms of the
L{\"{o}}wner order $\leq^{L}$: Let $A\beta$ be estimable. Then $Ly$ is said to
be BLUE of $A\beta$ if (i) $E(Ly)=A\beta$ for all $\beta\in\mathbb{R}^{p}$ and
(ii) $V(Ly)\leq^{L}V(My)$ for all $\beta\in\mathbb{R}^{p}$ and all $My$
satisfying $E(My)=A\beta$.

The second partial order which also has many applications in statistics (see
\cite[Sections 15.3, 15.4]{MitraKnjiga}) may be defined on the full set
$M_{m,n}(\mathbb{R})$. For $A,B\in M_{m,n}(\mathbb{R})$ we say that $A$ is
below $B$ with respect to the minus partial order (know also as the rank
substractivity partial order) and write
\[
A\leq^{-}B\quad\text{when}\quad A^{-}A=A^{-}B\text{ and }AA^{-}=BA^{-}%
\]
for some inner generalized inverse $A^{-}$ of $A$. It is known (see e.g.
\cite{MitraKnjiga}) that for $A,B\in M_{m,n}(\mathbb{R})$,
\begin{equation}
A\leq^{-}B\quad\text{if and only if}\quad\operatorname*{rank}%
(B-A)=\operatorname*{rank}(B)-\operatorname*{rank}(A). \label{eq_rank_minus}%
\end{equation}

Note that both orders may be defined in the same way on sets of complex
matrices \cite{MitraKnjiga}. Moreover, the minus partial order was introduced
by Hartwig in \cite{Hartwig} and independently by Nambooripad in
\cite{Nambooripad} on a general regular semigroup however it was mostly
studied on $M_{n}(\mathbb{F})$ (see \cite{Mitra} and the references therein).
More recently, \v{S}emrl generalized in \cite{Semrl} this order to
$B(\mathcal{H})$, the algebra of all bounded linear opearators on a Hilbert
space $\mathcal{H}$, and studied preservers of this order (see also
\cite{Legisa}). Let $\mathcal{A}$ be some subset of $B(\mathcal{H})$ and
denote by $\leq$ one of the above orders (i.e. $\leq^{L}$ or $\leq^{-}$). We
say that that a map $\varphi:\mathcal{A}\rightarrow\mathcal{A}$ preserves an
order $\leq$ in both directions when
\[
A\leq B\quad\text{if and only if}\quad\varphi(A)\leq\varphi(B)
\]
for every $A,B\in\mathcal{A}$.

Motivated by applications in quantum mechanics and quantum statistics
Moln\'{a}r studied preservers that are connected to certain structures of
bounded linear operators which appear in mathematical foundations of quantum
mechanics, i.e. he studied automorphisms of the underlying quantum structures
or, in other words, quantum mechanical symmetries. Let $A^{\ast}$ be the
adjoint operator of $A\in B(\mathcal{H})$, and let
\[
B^{+}(\mathcal{H})=\left\{  A\in B(\mathcal{H}):A=A^{\ast}\text{ and
}\left\langle Ax,x\right\rangle \geq0\text{ for every }x\in\mathcal{H}%
\right\}
\]
be the set of all positive operators in $B(\mathcal{H})$. Note that in case
when dim$\mathcal{H}<\infty$, the set $B^{+}(\mathcal{H})$ may be identified
with the set of all positive semidefinite $n\times n$ matrices. Note also that
we may generalize definition (\ref{Def_Lowner}) to the set of all
self-adjoined operators in $B(\mathcal{H})$ in the following way: For two
self-adjoined operators $A,B\in$ $B(\mathcal{H})$ we write $A\leq^{L}B$ when
$B-A\in B^{+}(\mathcal{H})$. Under assumption that $\mathcal{H}$ is a complex
Hilbert space with dim$\mathcal{H}>1$, Moln\'{a}r described in \cite{Molnar1}
the form of all bijective maps on $B^{+}(\mathcal{H})$ that preserve the
L{\"{o}}wner partial order in both directions. It turns out that every such a
map $\varphi$ is of the form
\begin{equation}
\varphi(A)=TAT^{\ast},\quad A\in B^{+}(\mathcal{H}) \label{Molnar_result}%
\end{equation}
where $T:$ $\mathcal{H\rightarrow H}$ is an invertible bounded either linear
or conjugate-linear operator. Since we expect that maps preserving the
L{\"{o}}wner order in both directions on the set of all real positive
semidefinite matrices may have interesting applications in statistics (e.g. in
the theory of comparison of linear models \cite{Stepniak}), we will study such
maps in Section 3. We will show that a similar result to Moln\'{a}r's Theorem
1 from \cite{Molnar1} holds also in the real matrix case, i.e. we will
characterize surjective maps (omitting the injectivity assumption) on the set
of all $n\times n$, $n\geq2$, positive semidefinite real matrices that
preserve the order $\leq^{L}$ in both directions. In Section 4, we will study
the minus partial order, search for applications of this order in statistics,
and describe the form of all surjective, additive maps on the set of all
$n\times n$, $n\geq3$, positive semidefinite real matrices that preserve the
minus partial order in both directions.

\section{Preliminaries}

Let us present some tools that will be useful throughout the paper. As before,
let $\mathbb{F=R}$ or $\mathbb{F=C}$. Let $H_{n}(\mathbb{F)}$ be the set of
all Hermitian (symmetric in the real case) matrices in $M_{n}(\mathbb{F})$,
denote by $H_{n}^{+}(\mathbb{F})$ the set of all positive semidefinite
matrices in $H_{n}(\mathbb{F)}$ and by $P_{n}(\mathbb{F})$ the set of
idempotent matrices in $H_{n}^{+}(\mathbb{F})$ (i.e. the set of all orthogonal
projection matrices in $M_{n}(\mathbb{F)}$). Let $V$ be a subspace of
$\mathbb{F}^{n}$. By $P_{V}\in P_{n}(\mathbb{F})$ we will denote the
orthogonal projection matrix with $\operatorname{Im}P_{V}=V$. Recall that a
convex cone $\mathcal{C}$ is a subset of a vector space $\mathcal{V}$ over an
ordered field that is closed under all linear combinations with nonnegative
scalars. For every convex cone $\mathcal{C}$, we will from now on assume that
$\mathcal{C}\cap(-\mathcal{C})=\{0\}$. Observe that then every convex cone
$\mathcal{C}$ induces a partial ordering $\leq$ on $\mathcal{V}$ so that we
write%
\[
x\leq y\quad\text{when}\quad y-x\in\mathcal{C}.
\]
Note that $H_{n}^{+}(\mathbb{F})$ is a convex cone which is closed in the real
normed vector space $H_{n}(\mathbb{F)}$. The following result of Rothaus
\cite{Rothaus} will be one of the main tools in the proof of our first theorem.

\begin{proposition}
\label{Prop_Rothaus}Let $\mathcal{D}$ be the interior of a closed convex cone
$\mathcal{C}$ in a real normed vector space $\mathcal{V}$. Suppose
$\varphi:\mathcal{D}$ $\rightarrow\mathcal{D}$ is a bijective map where
\[
x\leq y\quad\text{if and only if}\quad\varphi(x)\leq\varphi(y)
\]
for every $x,y\in\mathcal{D}$. Then the map $\varphi$ is linear.
\end{proposition}

We say that two Hermitian (symmetric) matrices $A,B\in M_{n}(\mathbb{F})$ are
adjacent if $\operatorname*{rank}(A-B)=1$. Huang and \v{S}emrl characterized
in \cite{Huang} maps $\varphi:$ $H_{n}(\mathbb{C)\rightarrow}H_{m}%
(\mathbb{C)}$, $m,n\in\mathbb{N}$, $n>1$, such that matrices $\varphi(A)$ and
$\varphi(B)$ are adjacent whenever $A$ and $B$ are adjacent, $A,B\in
H_{n}(\mathbb{C)}$. In \cite{Legisa1} Legi\v{s}a considered adjacency
preserving maps from $H_{n}(\mathbb{R)}$ to $H_{m}(\mathbb{R)}$ and proved the
following result.{}

\begin{proposition}
\label{Prop_Legisa}Let $n\geq2$ and let $\varphi:H_{n}(\mathbb{R)\rightarrow
}H_{m}(\mathbb{R)}$ be a map preserving adjacency, i.e. if $A,B\in
H_{n}(\mathbb{R)}$ and $\operatorname*{rank}(A-B)=1$, then
$\operatorname*{rank}(\varphi(A)-\varphi(B))=1$. Suppose $\varphi(0)=0$. Then either

\begin{itemize}
\item[(i)] there is a rank-one matrix $B\in H_{m}(\mathbb{R)}$ and a function
$f:H_{n}(\mathbb{R)\rightarrow R}$ such that for every $A\in H_{n}%
(\mathbb{R)}$
\[
\varphi(A)=f(A)B, or
\]

\item[(ii)] there exist $c\in\{-1,1\}$ and an invertible matrix $R\in
M_{m}(\mathbb{R)}$ such that for every $A\in H_{n}(\mathbb{R)}$
\[
\varphi(A)=cR\left[
\begin{array}
[c]{cc}%
A & 0\\
0 & 0
\end{array}
\right]  R^{t}.
\]
(Obviously, in this case $m\geq n.$ If $m=n$, the zeros on the right-hand side
of the formula are absent.)
\end{itemize}
\end{proposition}

We will conclude this section with an auxiliary result. Note first that for
$A,B\in H_{n}(\mathbb{F)}$, $B\leq^{L}A$ implies $\operatorname{Im}%
B\subseteq\operatorname{Im}A$ (see e.g. \cite[Corollary 8.2.12]{MitraKnjiga}).

\begin{lemma}
\label{Lemma_prelim}Let $A,B\in H_{n}^{+}(\mathbb{F})$ and let
$\operatorname*{rank}(A)=1$. If $B\leq^{L}A$, then $B=\lambda A$ for some
scalar $\lambda\in\left[  0,1\right]  $.
\end{lemma}

\begin{proof}
Since $A$ is of rank-one and $A\in H_{n}^{+}(\mathbb{F})$, it follows by the
spectral theorem \cite[page 46]{Conway} that $A=\alpha P$ where $\alpha>0$ and
$P\in P_{n}(\mathbb{F})$ with $\operatorname*{rank}(P)=1$. Let $B\leq^{L}A$
for some $B\in H_{n}^{+}(\mathbb{F})$. Then $\operatorname{Im}B\subseteq
\operatorname{Im}A$ and thus $\operatorname*{rank}(B)\leq1$. Again, by the
spectral theorem $B=\beta Q$ for some $\beta\geq0$ and a rank-one $Q\in
P_{n}(\mathbb{F})$. If $\beta=0$, then $B=0$ and thus $B=\lambda A$ for
$\lambda=0$. Suppose $\beta\neq0$. Since $\operatorname{Im}B\subseteq
\operatorname{Im}A$, we have $\operatorname{Im}Q=\operatorname{Im}P$ and thus
(since $P$ and $Q$ are orthogonal projection matrices) $P=Q$. Let
$\lambda=\frac{\beta}{\alpha}$. We have
\[
\lambda A=\frac{\beta}{\alpha}\alpha P=\beta P=B.
\]
Moreover, from $B\leq^{L}A$ it clearly follows that $\lambda\in\left[
0,1\right]  $.
\end{proof}

\section{Preservers of the L{\"{o}}wner partial order}

Let $S\in M_{n}(\mathbb{R})$ be an invertible matrix and $A,B,C\in
H_{n}(\mathbb{R)}$. It is easy to see (\cite[Theorem 8.2.7, Remark
8.2.8]{MitraKnjiga}) that then
\begin{equation}
A\leq^{L}B\quad\text{if and only if}\quad SAS^{t}\leq^{L}SBS^{t}.
\label{eq_S_invertible}%
\end{equation}
Also, if $A\leq^{L}B$, then $A+C\leq^{L}B+C$ and $\lambda A\leq^{L}\lambda B$
for every $\lambda\geq0$. Let us now state and prove our main result. The
proof will follow some ideas from \cite[the proof of Theorem 1]{Molnar1}
however for the sake of completeness and since we are dealing here with real
matrices, we will not skip the details and will
present it in its entirety.

\begin{theorem}
\label{Theorem_main:_Loewner}Let $n\geq2$ be an integer. Then $\varphi
:H_{n}^{+}(\mathbb{R})\rightarrow H_{n}^{+}(\mathbb{R})$ is a surjective map
that preserves the L{\"{o}}wner order $\leq^{L}$ in both directions if and
only if there exists an invertible matrix $S\in M_{n}(\mathbb{R)}$ such that
\[
\varphi(A)=SAS^{t}%
\]
for every $A\in H_{n}^{+}(\mathbb{R})$.
\end{theorem}

\begin{proof}
If $\varphi:H_{n}^{+}(\mathbb{R})\rightarrow H_{n}^{+}(\mathbb{R})$ is of the
form $\varphi(A)=SAS^{t}$, $A\in H_{n}^{+}(\mathbb{R})$, where $S\in
M_{n}(\mathbb{R)}$ is invertible, than it preserves by (\ref{eq_S_invertible})
the order $\leq^{L}$ in both directions and is clearly surjective.

Conversely, let $\varphi:H_{n}^{+}(\mathbb{R})\rightarrow H_{n}^{+}%
(\mathbb{R})$ be a surjective map that preserves the L{\"{o}}wner order
$\leq^{L}$ in both directions. We will split the proof into several steps.

1. $\varphi$ \textit{is bijective. }Let $\varphi(A)=\varphi(B)$ for $A,B\in
H_{n}^{+}(\mathbb{R})$. The order $\leq^{L}$ is reflexive so $\varphi
(A)\leq^{L}\varphi(B)$ and $\varphi(B)\leq^{L}\varphi(A)$. Since $\varphi$
preserves the order $\leq^{L}$ in both directions, we have $A\leq^{L}B$ and
$B\leq^{L}A$. It follows that $A=B$, since $\leq^{L}$ is antisymmetric. Thus,
$\varphi$ is injective and therefore bijective.

2. $\varphi(0)=0$. Note that $0\leq^{L}A$ for every $A\in H_{n}^{+}%
(\mathbb{R})$. So, on the one hand $0\leq^{L}\varphi(0)$ and on the other
hand, since $\varphi^{-1}$ has the same properties as $\varphi$, $0\leq
^{L}\varphi^{-1}(0)$ and thus $\varphi(0)\leq^{L}0$.

3. $\varphi$ \textit{preserves the set of all matrices of rank-one.} Let us
first show that $A\in H_{n}^{+}(\mathbb{R})$ is of rank-one if and only if for
every $B,C\in\left\{  D\in H_{n}^{+}(\mathbb{R}):D\leq^{L}A\right\}
\equiv\left[  0,A\right]  $ we have $B\leq^{L}C$ or $C\leq^{L}B,$ i.e. the
order $\leq^{L}$ is linear on $\left[  0,A\right]  $.

Let $A\in H_{n}^{+}(\mathbb{R})$ be of rank-one and suppose first
$B,C\in\left[  0,A\right]  $. By Lemma \ref{Lemma_prelim} we have $B=\lambda
A$ and $C=\mu A$ for some $\lambda,\mu\in\left[  0,1\right]  $. If $\lambda=0$
or $\mu=0$, then clearly $B\leq^{L}C$ or $C\leq^{L}B$. Suppose $\lambda\neq0$
and $\mu\neq0$. It follows that $\mu B=\lambda C$ and thus
\[
B-C=\left(  1-\frac{\mu}{\lambda}\right)  B.
\]
Clearly, then $0\leq^{L}B-C$ or $0\leq^{L}C-B$, i.e. $C\leq^{L}B$ or
$B\leq^{L}C$.

Conversely, suppose that the order $\leq^{L}$ is linear on $\left[
0,A\right]  $ and assume that $\operatorname*{rank}(A)>1$. By the spectral
theorem there exist $P_{1},P_{2}\in P_{n}(\mathbb{R})$ of rank-one with
$\operatorname{Im}P_{1}\cap\operatorname{Im}P_{2}=\left\{  0\right\}  $, and
$\lambda_{1},\lambda_{2}\in\left(  0,\infty\right)  $, such that $\lambda
_{1}P_{1}\leq^{L}A$ and $\lambda_{2}P_{2}\leq^{L}A$, i.e. $\lambda_{1}%
P_{1},\lambda_{2}P_{2}\in\left[  0,A\right]  .$ This yields by assumption
$\lambda_{1}P_{1}\leq^{L}\lambda_{2}P_{2}$ or $\lambda_{2}P_{2}$ $\leq
^{L}\lambda_{1}P_{1}$ and therefore in either case $\operatorname{Im}%
P_{1}=\operatorname{Im}P_{2}$, a contradiction.

Since $\varphi$ preserves the order $\leq^{L}$ in both directions, $\left[
0,A\right]  $ is linearly ordered if and only if $\left[  0,\varphi(A)\right]
$ is linearly ordered. Thus, $A\in H_{n}^{+}(\mathbb{R})$ is of rank-one if
and only if $\varphi(A)$ is of rank-one.

4. $\varphi$ \textit{preserves the set of all invertible (i.e. positive
definite) matrices. }For every matrix $P\in P_{n}(\mathbb{R})$ of rank $r$
there exists an orthogonal matrix $Q\in M_{n}(\mathbb{R})$ such that
\[
P=Q\left[
\begin{array}
[c]{cc}%
I_{r} & 0\\
0 & 0
\end{array}
\right]  Q^{t}%
\]
where $I_{r}$ is the $r\times r$ identity matrix. Let $I$ denote the identity
matrix in $M_{n}(\mathbb{R})$. Since then
\[
I-P=Q\left[
\begin{array}
[c]{cc}%
0 & 0\\
0 & I_{n-r}%
\end{array}
\right]  Q^{t}%
\]
it follows by the definition (\ref{Def_Lowner}) that $P\leq^{L}I$ for every
matrix $P\in P_{n}(\mathbb{R})$. This implies, $\epsilon P\leq^{L}\epsilon I$
for every $\varepsilon\geq0$. Let $\varepsilon>0$ be arbitrary but fixed. Let
us show that then $\varphi(\varepsilon I)$ is invertible. By the transitivity
of $\leq^{L}$, $\alpha P\leq^{L}\varepsilon I$ for every $P\in
P_{n}(\mathbb{R})$ and any scalar $\alpha$ where $0\leq\alpha\leq\varepsilon$.
Suppose $\varphi(\varepsilon I)$ is not invertible. Then there exists a
rank-one $Q\in P_{n}(\mathbb{R})$ such that $\operatorname{Im}Q\nsubseteq
\operatorname{Im}\varphi(\varepsilon I)$. Since $\varphi$ is surjective and
sends rank-one matrices to rank-one matrices, there exists a rank-one $P\in
P_{n}(\mathbb{R})$ and $\alpha>0$ such that $\varphi(\alpha P)=Q$. Here
$\alpha>\varepsilon$ since $\varphi$ preserves the order in both directions.
From $\varepsilon P\leq^{L}\alpha P$ we have $\varphi(\varepsilon P)\leq
^{L}\varphi(\alpha P)=Q$. Both $\varepsilon P$ and $Q$ are of rank-one and
therefore $\operatorname{Im}\varphi(\varepsilon P)=\operatorname{Im}Q$. This
is a contradiction since $\varphi(\varepsilon P)\leq^{L}\varphi(\varepsilon
I)$ and therefore $\operatorname{Im}\varphi(\varepsilon P)\subseteq
\operatorname{Im}\varphi(\varepsilon I)$. So, $\varphi(\varepsilon I)$ is
invertible for any $\varepsilon>0$.

Let now $T\in H_{n}^{+}(\mathbb{R})$ be an invertible (i.e. positive definite)
matrix. By \cite[page 93]{Pedersen} there exists $\varepsilon>0$ such that
$\varepsilon I\leq^{L}T$. It follows that $\varphi(\varepsilon I)\leq
^{L}\varphi(T)$ and thus $\mathbb{R}^{n}=\operatorname{Im}\varphi(\varepsilon
I)\subseteq\operatorname{Im}\varphi(T)$. So, $\varphi(T)$ is invertible. Since
$\varphi^{-1}$ has the same properties as $\varphi$, we may conclude that
$T\in H_{n}^{+}(\mathbb{R})$ is invertible if and only if $\varphi(T)$ is invertible.

5. $\varphi$ \textit{is linear on the set of all invertible matrices in
}$H_{n}^{+}(\mathbb{R})$. The interior of the set $H_{n}^{+}(\mathbb{R})$ of
all positive semidefinite matrices is the set of all invertible (i.e. positive
definite) matrices in $H_{n}^{+}(\mathbb{R})$ (see \cite[page 239]{Klerk}).
Since $H_{n}^{+}(\mathbb{R})$ is a convex cone which is closed in the real
normed vector space $H_{n}(\mathbb{R)}$ and since $\varphi$ preserves the set
of all invertible matrices, we may conclude by Proposition \ref{Prop_Rothaus}
that $\varphi$ is linear (additive and positive homogenous) on the set of all
invertible matrices in $H_{n}^{+}(\mathbb{R})$.

6. $\varphi$ \textit{is a linear map. }Let $A,B\in H_{n}^{+}(\mathbb{R})$ and
let $A_{k}=A+\frac{1}{k}I$, $B_{k}=B+\frac{1}{k}I$, $k\in\mathbb{N}$. Then
$\left\{  A_{k}\right\}  $ and $\left\{  B_{k}\right\}  $ are sequences of
positive definite (invertible) matrices in $H_{n}^{+}(\mathbb{R})$. Observe
that both sequences are monotone decreasing with respect to $\leq^{L}$ and
note that the sequence $\left\{  A_{k}\right\}  $ converges to $A$ and the
sequence $\left\{  B_{k}\right\}  $ converges to $B$ in the strong operator
topology. Also, $\inf_{k}A_{k}=A$ and $\inf_{k}B_{k}=B$ where $\inf$ denotes
the infimum of a sequence. We have $A+B=\inf_{k}(A_{k}+B_{k})$. Since
$\varphi$ preserves the order, it follows that $\varphi(A)=\inf_{k}%
\varphi(A_{k})$, $\varphi(B)=\inf_{k}\varphi(B_{k})$, and $\varphi
(A+B)=\inf_{k}\varphi(A_{k}+B_{k})$. Therefore, $\left\{  \varphi
(A_{k})\right\}  $, $\left\{  \varphi(B_{k})\right\}  $, and $\left\{
\varphi(A_{k}+B_{k})\right\}  $ are monotone decreasing sequences bounded from
below. By \cite[Definition 2.8 and Example 2.10]{Burbanks} (see also
\cite[page 263]{Riesz}) there exist limits (in the strong sense) of these
sequences that equal their infima. Thus,
\[
\varphi(A)=\underset{k\rightarrow\infty}{\lim}\varphi(A_{k}),\quad
\varphi(B)=\underset{k\rightarrow\infty}{\lim}\varphi(B_{k}),\quad
\varphi(A+B)=\underset{k\rightarrow\infty}{\lim}\varphi(A_{k}+B_{k}).
\]
Step 5 yields that $\varphi(A_{k}+B_{k})=\varphi(A_{k}%
)+\varphi(B_{k})$ and hence
\[
\varphi(A+B)=\underset{k\rightarrow\infty}{\lim}\varphi(A_{k}+B_{k}%
)=\underset{k\rightarrow\infty}{\lim}\varphi(A_{k})+\underset{k\rightarrow
\infty}{\lim}\varphi(B_{k})=\varphi(A)+\varphi(B),
\]
i.e. $\varphi$ is additive. To show that $\varphi$ is also (positive)
homogenous, let $\lambda\geq0$ be any scalar. Clearly, $\lambda A=\inf
_{k}(\lambda A_{k})$. Again, by the previous step it follows that
\[
\varphi(\lambda A)=\underset{k\rightarrow\infty}{\lim}\varphi(\lambda
A_{k})=\lambda\underset{k\rightarrow\infty}{\lim}\varphi(A_{k})=\lambda
\varphi(A).
\]

7. \textit{We will extend the map }$\varphi$ \textit{from} $H_{n}%
^{+}(\mathbb{R})$ \textit{to} $H_{n}(\mathbb{R)}$. Let $A\in H_{n}%
(\mathbb{R)}$. There exists an orthogonal matrix $Q\in M_{n}(\mathbb{R})$ such
that $A=Q^{t}DQ$ where $D$ is a diagonal matrix having the eigenvalues of $A$
on the diagonal, i.e. $D=$diag$\left(  \lambda_{i}:1\leq i\leq n\right)  $.
Let $D^{+}=$diag$\left(  \lambda_{i}^{+}:1\leq i\leq n\right)  $ and $D^{-}%
=$diag$\left(  \lambda_{i}^{-}:1\leq i\leq n\right)  $ where $\lambda_{i}%
^{+}=\max\left\{  \lambda_{i},0\right\}  $ and $\lambda_{i}^{-}=\max\left\{
-\lambda_{i},0\right\}  $. Clearly, then $A=Q^{t}D^{+}Q-Q^{t}D^{-}Q$. Note
that both $Q^{t}D^{+}Q,Q^{t}D^{-}Q\in H_{n}^{+}(\mathbb{R})$. We call the
matrices $Q^{t}D^{+}Q$ and $Q^{t}D^{-}Q$ the positive and the negative part of
$A$, respectively. We may now extend the map $\varphi$ to the map
$\widehat{\varphi}:H_{n}(\mathbb{R)\rightarrow}H_{n}(\mathbb{R)}$ in the
following way:%
\[
\widehat{\varphi}(C)=\varphi(C^{+})-\varphi(C^{-}),\quad C\in H_{n}%
(\mathbb{R)},
\]
where $C^{+}$ and $C^{-}$ are the positive and the negative part of $C$,
respectively. Recall that $\varphi(0)=0$. Take $C\in H_{n}^{+}(\mathbb{R})$
and note that then $C^{+}=C$ and $C^{-}=0$. So, $\widehat{\varphi}%
(C)=\varphi(C)-\varphi(0)=\varphi(C)$.

8. $\widehat{\varphi}$ \textit{is a linear map. }Let $A,B\in H_{n}%
^{+}(\mathbb{R})$ and $C=A-B$. So, $C\in H_{n}(\mathbb{R})$. From $C^{+}%
-C^{-}=C=A-B$, we have $C^{+}+B=A+C^{-}\in H_{n}^{+}(\mathbb{R})$. Recall that
$\varphi$ is additive hence $\varphi(C^{+})+\varphi(B)=\varphi(A)+\varphi
(C^{-})$ and thus
\begin{equation}
\widehat{\varphi}(A-B)=\widehat{\varphi}(C)=\varphi(C^{+})-\varphi
(C^{-})=\varphi(A)-\varphi(B). \label{eq_proof_linear}%
\end{equation}
Let us show that $\widehat{\varphi}$ is additive. Let $C,D\in$ $H_{n}%
(\mathbb{R})$. Then by (\ref{eq_proof_linear})
\begin{align*}
\widehat{\varphi}(C+D)  &  =\widehat{\varphi}(C^{+}-C^{-}+D^{+}-D^{-}%
)=\widehat{\varphi}((C^{+}+D^{+})-(C^{-}+D^{-}))\\
&  =\varphi(C^{+}+D^{+})-\varphi(C^{-}+D^{-})=\varphi(C^{+})-\varphi
(C^{-})+\varphi(D^{+})-\varphi(D^{-})\\
&  =\widehat{\varphi}(C)+\widehat{\varphi}(D).
\end{align*}

Let us now prove that $\widehat{\varphi}$ is homogenous. Let $C\in$
$H_{n}(\mathbb{R})$ and let $\lambda\in\mathbb{R}$. Suppose first $\lambda
\geq0$. Then $(\lambda C)^{+}=\lambda C^{+}$ and $(\lambda C)^{-}=\lambda
C^{-}$ are the positive and the negative part of $\lambda C$, respectively.
Since $\varphi$ is (positive) homogenous, we have
\[
\widehat{\varphi}(\lambda C)=\varphi(\lambda C^{+})-\varphi(\lambda
C^{-})=\lambda\varphi(C^{+})-\lambda\varphi(C^{-})=\lambda\widehat{\varphi
}(C).
\]
Let now $\lambda<0$. Then $(\lambda C)^{+}=-\lambda C^{-}$ and $(\lambda
C)^{-}=-\lambda C^{+}$. So, $\widehat{\varphi}(\lambda C)=\widehat{\varphi
}(-\lambda C^{-}-(-\lambda C^{+}))$ and therefore
by (\ref{eq_proof_linear})
\[
\widehat{\varphi}(\lambda C)=\varphi(-\lambda C^{-})-\varphi(-\lambda
C^{+})=-\lambda\varphi(C^{-})-(-\lambda)\varphi(C^{+})=\lambda(\varphi
(C^{+})-\varphi(C^{-}))=\lambda\widehat{\varphi}(C).
\]

9. $\widehat{\varphi}$ \textit{preserves the order }$\leq^{L}$ \textit{in both
directions. }Since $\widehat{\varphi}(C)=\varphi(C)$ for every $C\in H_{n}%
^{+}(\mathbb{R})$, we observe that $0\leq^{L}C$ if and only if $0\leq^{L}$
$\widehat{\varphi}(C)$. Let $C_{1},C_{2}\in H_{n}(\mathbb{R})$. Then
$C_{1}\leq^{L}C_{2}$ if and only if $0\leq^{L}$ $\widehat{\varphi}(C_{2}%
-C_{1})$. Since $\widehat{\varphi}$ is linear, this equivalent to
$\widehat{\varphi}(C_{1})\leq^{L}\widehat{\varphi}(C_{2})$.

10. $\widehat{\varphi}$ \textit{is bijective. }Since $\widehat{\varphi}$
preserves the order $\leq^{L}$ in both directions, it is clearly
injective (see the first step). To show that $\widehat{\varphi}$ is
surjective, let $C\in H_{n}(\mathbb{R})$. Then we may write $C=C^{+}-C^{-}$
where $C^{+},C^{-}\in H_{n}^{+}(\mathbb{R})$. Since $\varphi$ is surjective,
there exist $A,B\in H_{n}^{+}(\mathbb{R})$ such that $C^{+}=\varphi
(A)=\widehat{\varphi}(A)$ and $C^{-}=\varphi(B)=\widehat{\varphi}(B)$. So,
\[
C=C^{+}-C^{-}=\widehat{\varphi}(A)-\widehat{\varphi}(B)=\widehat{\varphi
}(A-B),
\]
i.e. $\widehat{\varphi}$ is surjective.

11. $\widehat{\varphi}$ \textit{is an adjacency preserving map}. Let us first
show that $\widehat{\varphi}$ preserves the set of all rank-one matrices. Let
$C\in H_{n}(\mathbb{R})$ be a rank-one matrix. By the spectral theorem,
$C=\alpha P$ where $\alpha\in\mathbb{R}$ is nonzero and $P\in P_{n}%
(\mathbb{R})$ is of rank-one. Since $\widehat{\varphi}$ is linear and since
$P\in H_{n}^{+}(\mathbb{R})$, we have
\[
\widehat{\varphi}(C)=\alpha\widehat{\varphi}(P)=\alpha\varphi(P).
\]
Recall that $\varphi$ preserves the set of rank-one matrices. It follows that
$\widehat{\varphi}(C)$ is of rank-one. Let now $A,B\in H_{n}(\mathbb{R})$ with
$\operatorname*{rank}(A-B)=1$, i.e. let $A$ and $B$ be adjacent. It follows
that $\widehat{\varphi}(A-B)$ is of rank-one. Since $\widehat{\varphi
}(A-B)=\widehat{\varphi}(A)-\widehat{\varphi}(B)$, we may conclude that
$\widehat{\varphi}(A)$ and $\widehat{\varphi}(B)$ are adjacent.

We are now in the position to conclude the proof of the theorem. Since
$\widehat{\varphi}:$ $H_{n}(\mathbb{R)\rightarrow}H_{n}(\mathbb{R)}$ is a
bijective map that preserves adjacency, it follows by Proposition
\ref{Prop_Legisa} that there exists $c\in\left\{  -1,1\right\}  $ and an
invertible $S\in M_{n}(\mathbb{R})$ such that
\[
\widehat{\varphi}(A)=cSAS^{t},\quad A\in H_{n}(\mathbb{R)}.
\]
Let $A,B\in H_{n}(\mathbb{R)}$, $A\neq B$, and $A\leq^{L}B$. Then on the one
hand by (\ref{eq_S_invertible}), $SAS^{t}\leq^{L}SBS^{t}$. If $c=-1$, we get
on the one hand, since $\widehat{\varphi}$ preserves the order $\leq^{L}$,
$-SAS^{t}\leq^{L}-SBS^{t}$. It follows that $SAS^{t}=SBS^{t}$ and therefore
$A=B$, a contradiction. To conclude, $\widehat{\varphi}(A)=SAS^{t}$ for every
$A\in H_{n}(\mathbb{R)}$ and therefore $\varphi(A)=SAS^{t}$ for every $A\in
H_{n}^{+}(\mathbb{R})$.
\end{proof}

\begin{remark}
The proof of Theorem \ref{Theorem_main:_Loewner} may serve with a few
adjustments (e.g. instead of Proposition \ref{Prop_Legisa} we may use Theorem
1.2 from \cite{Huang} (see also \cite{Hua1,Hua2})) as an alternative proof of
finite-dimensional (complex) version (dim$\mathcal{H}<\infty$) of Moln\'{a}r's
result (\ref{Molnar_result}).
\end{remark}

\begin{remark}
Let us present an observation about preservers of the L{\"{o}}wner partial
order and linear models. Let $L_{1}=$ $(y_{1},X_{1}\beta,\sigma^{2}D_{1})$ and
$L_{2}=$ $(y_{2},X_{2}\beta,\sigma^{2}D_{2})$ be two linear models. Here
$X_{1}\in M_{n,p}(\mathbb{R})$, $X_{2}\in M_{m,p}(\mathbb{R})$, $D_{1}\in
H_{n}^{+}(\mathbb{R)}$, and $D_{2}\in H_{m}^{+}(\mathbb{R)}$. We say (see
\cite{Stepniak2}) that $L_{1}$ is at least as good as $L_{2}$ if for any
unbiased estimator $a_{2}^{t}y_{2}$, $a_{2}\in M_{m,1}(\mathbb{R})$, of a
parameter $k^{t}\beta$, $k\in M_{p,1}(\mathbb{R})$, there exists an unbiased
estimator $a_{1}^{t}y_{1}$, $a_{1}\in M_{n,1}(\mathbb{R})$, such that
$V(a_{1}^{t}y_{1})\leq^{L}V(a_{2}^{t}y_{2})$ (here $V(a_{i}^{t}y_{i})$,
$i\in\left\{  1,2\right\}  $, is the variance of $a_{i}^{t}y_{i}$). If this
condition is satisfied, we write $L_{1}\succeq L_{2}$. In \cite{Stepniak},
St\k{e}pniak proved that
\[
L_{1}\succeq L_{2}\quad\text{if and only if}\quad M_{2}\leq^{L}M_{1}%
\]
where $M_{i}=X_{i}^{t}\left(  D_{i}+X_{i}X_{i}^{t}\right)  ^{-}X_{i}$,
$i\in\left\{  1,2\right\}  $. Moreover, St\k{e}pniak noted that when
$\operatorname{Im}X_{i}\subseteq\operatorname{Im}D_{i}$, $i\in\left\{
1,2\right\}  $, we may replace $X_{i}^{t}\left(  D_{i}+X_{i}X_{i}^{t}\right)
^{-}X_{i}$ with $X_{i}^{t}D_{i}^{.-}X_{i}$. When $D_{i}=X_{i}$, $i\in\left\{
1,2\right\}  $, these matrices may be further simplified to $M_{i}=X_{i}%
^{t}D_{i}^{.-}X_{i}=D_{i}^{t}D_{i}^{-}D_{i}=D_{i}D_{i}^{-}D_{i}=D_{i}$. For
such models $L_{1}=(y_{1},D_{1}\beta,\sigma^{2}D_{1})$ and $L_{2}=$
$(y_{2},D_{2}\beta,\sigma^{2}D_{2})$ we thus have
\begin{equation}
L_{1}\succeq L_{2}\quad\text{if and only if}\quad D_{2}\leq^{L}D_{1}.
\label{eq_order_models}%
\end{equation}
Let $n>1$. For a random $n\times1$ vector of observed quanitities $y_i$, an unspecified $n\times1$ vector
$\beta_i$, and an unspecified nonnegative scalar $\sigma^{2}_i$, let $\mathcal{L}_i$ be the set of all linear models $L_i=(y_i,D\beta_i
,\sigma^{2}_i D)$ where $D\in H_{n}^{+}(\mathbb{R)}$ may vary from model to
model. Define a map $\psi:\mathcal{L}_1\rightarrow \mathcal{L}_2$ with $\psi((y_1,D\beta_1,\sigma_{1}^{2}D))=
(y_2,\varphi(D)\beta_2,\sigma_{2}^{2}\varphi(D))$ where $\varphi:H_{n}^{+}(\mathbb{R)}$
$\rightarrow H_{n}^{+}(\mathbb{R)}$ is a surjective map. Suppose
\[
L_{1_a}\succeq L_{1_b}\quad\text{if and only if}\quad\psi(L_{1_a})\succeq\psi
(L_{1_b})
\]
for every $L_{1_a},L_{1_b}\in\mathcal{L}_1$. This assumption may be reformulated as $D_{1_b}%
\leq^{L}D_{1_a}$ if and only if $\varphi(D_{1_b})\leq^{L}\varphi(D_{1_a})$,
$D_{1_a},D_{1_b}\in$ $H_{n}^{+}(\mathbb{R)}$, and therefore Theorem
\ref{Theorem_main:_Loewner} completely determines the form of any such a map
$\psi$.
\end{remark}

\section{Preservers of the minus partial order}

Let $A,B\in M_{n}(\mathbb{F)}$. It is known (see e.g. \cite[page 149]{Legisa})
that
\begin{equation}
A\leq^{-}B\quad\text{if and only if}\quad\operatorname{Im}B=\operatorname{Im}%
A\oplus\operatorname{Im}(B-A)\quad\text{if and only if}\quad RAL\leq^{-}RBL
\label{eq_minus}%
\end{equation}
for any invertible $R,L\in M_{n}(\mathbb{F)}$. Let $A,B\in M_{n}(\mathbb{C)}$.
If there exists an invertible matrix $S\in M_{n}(\mathbb{C)}$ such that

\begin{itemize}
\item[a)] $B=$ $SAS^{t}$, then we say that $A$ and $B$ are congruent;

\item[b)] $B=$ $SAS^{\ast}$, then we say that $A$ and $B$ are *congruent.
\end{itemize}

By Sylvester's law of inertia (see \cite[page 282]{Horn}) two (Hermitian)
matrices $A,B\in H_{n}(\mathbb{C)}$ are *congruent if and only if they have
the same inertia, i.e. they have the same number of positive eigenvalues and
the same number of negative eigenvalues. Two (real symmetric) matrices $A,B\in
H_{n}(\mathbb{R)}$ are *congruent via a complex matrix if and only if they are
congruent via a real matrix \cite[page 283]{Horn}. So, Sylvester's law for the
real case states that $A,B\in H_{n}(\mathbb{R)}$ are congruent via an
invertible $S\in M_{n}(\mathbb{R)}$ (i.e. $B=$ $SAS^{t}$) if and only if $A$
and $B$ have the same number of positive eigenvalues and the same number of
negative eigenvalues. Note that congruent (respectively, *congruent) matrices
have the same rank \cite[page 281]{Horn}.

The next theorem gives a characerization of the minus partial order on the
cone of all positive semidefinite matrices. Observe first that if $A$ is an
$n\times n$ zero matrix, then $A\leq^{-}B$ for every $B\in M_{n}(\mathbb{F)}$
(see e.g. (\ref{eq_minus})).

\begin{theorem}
\label{Theorem_charact_minus}Let $A,B\in H_{n}^{+}(\mathbb{F)}$ and $A\neq0$.
Then $A\leq^{-}B$ if and only if there exists an invertible matrix $S\in M_{n}$
such that
\[
A=S\left[
\begin{array}
[c]{cc}%
I_{r} & 0\\
0 & 0
\end{array}
\right]  S^{\ast}\quad\text{and}\quad B=S\left[
\begin{array}
[c]{cc}%
I_{s} & 0\\
0 & 0
\end{array}
\right]  S^{\ast}%
\]
where $I_{r}$ and $I_{s}$ are $r\times r$ and $s\times s,$ $s\leq n$, identity
matrices, respectively, and $r<s$ if $A\neq B$, and $r=s$, otherwise.\newline%
(Obviously, in case when $s=n$, the zeros on the right-hand side of the
formula for $B$ are absent.)
\end{theorem}

\begin{proof}
To simplify notation we will use the term *congruent for both *congruent
complex matrices (via an invertible complex matrix) and congruent real
matrices (via a real invertible matrix). Of course, $S^{\ast}=S^{t}$ when
$S\in M_{n}(\mathbb{R)}$.

Let $A\in H_{n}^{+}(\mathbb{F)}$, $A\neq0$. Suppose $A\leq^{-}B$ for some
$B\in H_{n}^{+}(\mathbb{F)}$. By (\ref{eq_rank_minus}), $\operatorname*{rank}%
(B-A)=\operatorname*{rank}(B)-\operatorname*{rank}(A)$. Let $C=B-A$. So,
$\operatorname*{rank}(C)+\operatorname*{rank}(A)=\operatorname*{rank}(A+C)$.
Observe that $A+C$ is positive semidefinite (because $B$ is). All the
eigenvalues of the matrix $A+C$ are thus nonnegative and therefore by
Sylvester's law of inertia it follows that there exists an invertible matrix
$V\in M_{n}(\mathbb{F)}$ such that
\[
V(A+C)V^{\ast}=\left[
\begin{array}
[c]{cc}%
I_{s} & 0\\
0 & 0
\end{array}
\right]
\]
where $I_{s}$ is an $s\times s$, $s\leq n$, identity matrix. Let
\begin{equation}
Q=\left[
\begin{array}
[c]{cc}%
I_{s} & 0\\
0 & 0
\end{array}
\right]  ,\quad A_{1}=VAV^{\ast},\quad\text{and\quad}C_{1}=VCV^{\ast}.
\label{Q_A_1_C_1}%
\end{equation}
Since *congruent matrices have the same rank, it follows that
$\operatorname*{rank}(A+C)=\operatorname*{rank}(Q)$, $\operatorname*{rank}%
(A)=\operatorname*{rank}(A_{1})$, $\operatorname*{rank}%
(C)=\operatorname*{rank}(C_{1})$, and therefore
\begin{equation}
\operatorname*{rank}(Q)=\operatorname*{rank}(A_{1})+\operatorname*{rank}%
(C_{1}). \label{eq_char_1}%
\end{equation}
Observe that
\begin{equation}
\operatorname*{Im}Q=\operatorname*{Im}(V(A+C)V^{\ast})=\operatorname*{Im}%
(VAV^{\ast}+VCV^{\ast})\subseteq\operatorname*{Im}(VAV^{\ast}%
)+\operatorname*{Im}(VCV^{\ast}). \label{eq_char_2}%
\end{equation}
By (\ref{eq_char_1}) and (\ref{eq_char_2}) we have $\operatorname*{Im}%
Q=\operatorname*{Im}A_{1}+\operatorname*{Im}C_{1}$. Also, if
$\operatorname*{Im}A_{1}\cap\operatorname*{Im}C_{1}\neq\left\{  0\right\}  $,
then $\operatorname*{rank}(A_{1})+\operatorname*{rank}(C_{1}%
)>\operatorname*{rank}(Q)$, a contradiction. Thus,
\begin{equation}
\operatorname*{Im}Q=\operatorname*{Im}A_{1}\oplus\operatorname*{Im}%
C_{1}.\newline\label{eq_im}%
\end{equation}
Let $x\in\operatorname*{Ker}Q$, i.e. $Qx=0$. From $Q=A_{1}+C_{1}$, we have
$0=Qx=A_{1}x+C_{1}x$. Since $0=0+0$, it follows by (\ref{eq_im}) that
$A_{1}x=0$ and $C_{1}x=0$. So, $A_{1}(\operatorname*{Ker}Q)=\left\{
0\right\}  $ and $C_{1}(\operatorname*{Ker}Q)=\left\{  0\right\}  $. The
matrix $Q$ is clearly a self-adjoined idempotent, i.e. $Q^{\ast}=Q=Q^{2}$.
So,
\[
\mathbb{F}^{n}=\operatorname*{Im}Q\oplus\operatorname*{Ker}Q
\]
where $(\operatorname{Im}Q)^{\perp}=\operatorname*{Ker}Q$.

Consider the representation of a linear operator $D:\mathbb{F}^{n}%
\rightarrow\mathbb{F}^{n}$ with respect to the decomposition $\mathbb{F}%
^{n}=\operatorname*{Im}Q\oplus\operatorname*{Ker}Q:$
\[
D=\left[
\begin{array}
[c]{cc}%
D_{1} & D_{2}\\
D_{3} & D_{4}%
\end{array}
\right]
\]
where $D_{1}:\operatorname{Im}Q\rightarrow\operatorname{Im}Q$, $D_{2}:$
$\operatorname*{Ker}Q\rightarrow\operatorname{Im}Q$, $D_{3}:$
$\operatorname{Im}Q\rightarrow\operatorname*{Ker}Q$, and $D_{4}%
:\operatorname*{Ker}Q\rightarrow\operatorname*{Ker}Q$ are linear operators.
Since we may consider $A_{1}$ and $C_{1}$ as operators from
$\operatorname*{Im}Q\oplus\operatorname*{Ker}Q$ to itself, we may conclude
that with respect to this decomposition
\[
A_{1}=\left[
\begin{array}
[c]{cc}%
\widetilde{A_{1}} & 0\\
\widetilde{A_{2}} & 0
\end{array}
\right]  \quad\text{and}\quad C_{1}=\left[
\begin{array}
[c]{cc}%
\widetilde{C_{1}} & 0\\
\widetilde{C_{2}} & 0
\end{array}
\right]  .
\]
Observe that $A_{1}^{\ast}=(VAV^{\ast})^{\ast}=VA^{\ast}V^{\ast}=VAV^{\ast
}=A_{1}$. Similarly, $C_{1}^{\ast}=C_{1}$ and hence it follows that
$\widetilde{A_{2}}=0$ and $\widetilde{C_{2}}=0$, i.e.
\[
A_{1}=\left[
\begin{array}
[c]{cc}%
\widetilde{A_{1}} & 0\\
0 & 0
\end{array}
\right]  \quad\text{and}\quad C_{1}=\left[
\begin{array}
[c]{cc}%
\widetilde{C_{1}} & 0\\
0 & 0
\end{array}
\right]  .
\]
Since $\operatorname*{rank}(Q)=s$ (see (\ref{Q_A_1_C_1})), it follows by
(\ref{eq_im}) that
\begin{equation}
\mathbb{F}^{s}=\operatorname*{Im}\widetilde{A_{1}}\oplus\operatorname*{Im}%
\widetilde{C_{1}}. \label{eq_direct}%
\end{equation}
Note that $Qx=x$ for every $x\in\operatorname*{Im}Q$. Let $x\in
\operatorname*{Im}\widetilde{A_{1}}$. On the one hand $x=\widetilde{A_{1}%
}x+\widetilde{C_{1}}x$ and on the other hand $x=x+0$. By (\ref{eq_direct}) it
follows $x=\widetilde{A_{1}}x$ and $0=\widetilde{C_{1}}x$. Let now
$x\in\operatorname*{Im}\widetilde{C_{1}}$. Similarly, then $x=\widetilde
{A_{1}}x+\widetilde{C_{1}}x$ and $x=0+x$ and therefore $0=\widetilde{A_{1}}x$
and $\widetilde{C_{1}}x=x$. So, $\widetilde{A_{1}}$ acts as the identity
operator on $\operatorname*{Im}\widetilde{A_{1}}$ and as the zero operator on
$\operatorname*{Im}\widetilde{C_{1}}$, and similarly, $\widetilde{C_{1}}$ acts
as the identity operator on $\operatorname*{Im}\widetilde{C_{1}}$ and as the
zero operator on $\operatorname*{Im}\widetilde{A_{1}}$. This yields by
(\ref{eq_direct}) that $\operatorname*{Im}\widetilde{A_{1}}=$
$\operatorname*{Ker}\widetilde{C_{1}}$ and $\operatorname*{Ker}\widetilde
{A_{1}}=\operatorname*{Im}\widetilde{C_{1}}$. It follows
that $\widetilde{A_{1}}$ and $\widetilde{C_{1}}$ are pairwise orthogonal
idempotent operators on $\mathbb{F}^{s}$, and therefore $\widetilde{A_{1}}$ and $\widetilde{C_{1}}$ are
simultaneously diagonalizable (see e.g. \cite{Jacobson}). Recall that both
$\widetilde{A_{1}}$ and $\widetilde{C_{1}}$ are self-adjoined. It follows that
there exists a unitary (i.e. an orthogonal in the real case) matrix $U\in
M_{s}(\mathbb{F})$ such that
\[
U\widetilde{A_{1}}U^{\ast}=\left[
\begin{array}
[c]{cc}%
I_{r} & 0\\
0 & 0
\end{array}
\right]  \quad\text{and}\quad U\widetilde{C_{1}}U^{\ast}=\left[
\begin{array}
[c]{cc}%
0 & 0\\
0 & I_{s-r}%
\end{array}
\right]
\]
where $I_{r}$ and $I_{s-r}$ are $r\times r$ and $(s-r)\times(s-r)$ identity
matrices. Let
\[
Z=\left[
\begin{array}
[c]{cc}%
U & 0\\
0 & I_{n-s}%
\end{array}
\right]  .
\]
Note that $Z\in M_{n}(\mathbb{F)}$ is invertible. Then
\[
ZA_{1}Z^{\ast}=\left[
\begin{array}
[c]{cc}%
U & 0\\
0 & I_{n-s}%
\end{array}
\right]  \left[
\begin{array}
[c]{cc}%
\widetilde{A_{1}} & 0\\
0 & 0
\end{array}
\right]  \left[
\begin{array}
[c]{cc}%
U^{\ast} & 0\\
0 & I_{n-s}%
\end{array}
\right]  =\left[
\begin{array}
[c]{cc}%
U\widetilde{A_{1}}U^{\ast} & 0\\
0 & 0
\end{array}
\right]  =\left[
\begin{array}
[c]{ccc}%
I_{r} & 0 & 0\\
0 & 0 & 0\\
0 & 0 & 0
\end{array}
\right]  .
\]
Similarly,
\[
Z\widetilde{C_{1}}Z^{\ast}=\left[
\begin{array}
[c]{ccc}%
0 & 0 & 0\\
0 & I_{s-r} & 0\\
0 & 0 & 0
\end{array}
\right]  .
\]
Let $S=(ZV)^{-1}$. Then by (\ref{Q_A_1_C_1}),
\[
A=V^{-1}A_{1}(V^{\ast})^{-1}=V^{-1}Z^{-1}\left[
\begin{array}
[c]{ccc}%
I_{r} & 0 & 0\\
0 & 0 & 0\\
0 & 0 & 0
\end{array}
\right]  (Z^{\ast})^{-1}(V^{\ast})^{-1}=S\left[
\begin{array}
[c]{ccc}%
I_{r} & 0 & 0\\
0 & 0 & 0\\
0 & 0 & 0
\end{array}
\right]  S^{\ast}.
\]
Similarly,
\[
C=S\left[
\begin{array}
[c]{ccc}%
0 & 0 & 0\\
0 & I_{s-r} & 0\\
0 & 0 & 0
\end{array}
\right]  S^{\ast}%
\]
and therefore
\[
B=A+C=S\left[
\begin{array}
[c]{ccc}%
I_{r} & 0 & 0\\
0 & I_{s-r} & 0\\
0 & 0 & 0
\end{array}
\right]  S^{\ast}.
\]
\newline So,
\[
A=S\left[
\begin{array}
[c]{cc}%
I_{r} & 0\\
0 & 0
\end{array}
\right]  S^{\ast}\quad\text{and}\quad B=S\left[
\begin{array}
[c]{cc}%
I_{s} & 0\\
0 & 0
\end{array}
\right]  S^{\ast}%
\]
where $r\leq s$. Clearly, if $A\neq B$, then $r<s$, and $r=s$, otherwise.

Conversely, let $A=S\left[
\begin{array}
[c]{cc}%
I_{r} & 0\\
0 & 0
\end{array}
\right]  S^{\ast}$ and $B=S\left[
\begin{array}
[c]{cc}%
I_{s} & 0\\
0 & 0
\end{array}
\right]  S^{\ast}$ where $r\leq s$. It follows that
\[
B-A=S\left[
\begin{array}
[c]{ccc}%
0 & 0 & 0\\
0 & I_{s-r} & 0\\
0 & 0 & 0
\end{array}
\right]  S^{\ast}.
\]
Since congruence preserves rank, we have $\operatorname*{rank}%
(B-A)=\operatorname*{rank}(B)-\operatorname*{rank}(A)$ and therefore
$A\leq^{-}B$.
\end{proof}

As an example of an application of the minus partial order in statistics we present the following two corollaries to Theorem \ref{Theorem_charact_minus}. The first result is a
direct corollary to Theorem \ref{Theorem_charact_minus} and the main result
in \cite[page 366]{Baksalary}.

\begin{corollary}
Consider a linear model $(y,X\beta,\sigma^{2}D)$. Then the statistics $Ly$
with $V(Ly)\neq V(y)$ is BLUE of $X\beta$ if and only if the following
conditions hold:
\begin{itemize}
\item[(i)] $LX=X$;

\item[(ii)] $\operatorname*{Im}(LD)\subseteq\operatorname*{Im}X$;

\item[(iii)] There exist an invertible matrix $S\in M_{n}(\mathbb{R)}$ such
that
\[
V(Ly)=S\left[
\begin{array}
[c]{cc}%
I_{r} & 0\\
0 & 0
\end{array}
\right]  S^{t}\quad\text{and}\quad V(y)=S\left[
\begin{array}
[c]{cc}%
I_{s} & 0\\
0 & 0
\end{array}
\right]  S^{t}%
\]
where $I_{r}$ is a $r\times r$ identity matrix, and $I_{s}$ is a $s\times s$
identity matrix with $r<s\leq n$.
\end{itemize}
\end{corollary}

Note that for a positive semidefinite matrix
$A\in M_{n}(\mathbb{R)}$, the matrix $W^{t}AW\in M_{m}(\mathbb{R})$ is still
positive semidefinite for any matrix $W\in M_{n,m}(\mathbb{F})$. The following result thus follows directly from
Theorem \ref{Theorem_charact_minus} and \cite[Theorem 1]{BaksalaryHaukeStyan}.

\begin{corollary}
Let $A=\sum\nolimits_{i=1}^{k}A_{i}$ where $A_{i}\in M_{n}(\mathbb{R})$ are
positive semidefinite matrices, $i=1,2,\ldots ,k$. Let the $n\times 1$
random vector $x$ follow a multivariate normal distribution with the mean $%
\mu $ and the variance-covariance matrix $V$. Let $W=(V:\mu )$ be a $n\times
(n+1)$ partitioned matrix. Consider the quadratic forms $Q=x^{t}Ax$ and $%
Q_{i}=x^{t}A_{i}x$, $i=1,2,\ldots ,k$. Then the following statements are
equivalent.
\begin{itemize}
\item[(i)] $Q_{i}$, $i=1,2,\ldots ,k$, are mutually independent and
distributed as chi-squared variables;
\item[(ii)] $Q$ is distributed as a chi-squared variable and there exist
invertible matrices $S_{i}\in M_{n+1}(\mathbb{R)}$ such that
\begin{equation*}
W^{t}A_{i}W=S_{i}\left[
\begin{array}{cc}
I_{r_{i}} & 0 \\
0 & 0%
\end{array}%
\right] S_{i}^{t}\quad \text{and}\quad W^{t}AW=S_{i}\left[
\begin{array}{cc}
I_{s} & 0 \\
0 & 0%
\end{array}%
\right] S_{i}^{t}
\end{equation*}%
for every $i=1,2,\ldots ,k$, where $I_{r_{i}}$ are $r_{i}\times r_{i}$
identity matrices, and $I_{s}$ is a $s\times s$ identity matrix with $%
r_{i}\leq s\leq n+1$. (Here $I_{r_{i}}=0$ if $W^{t}A_{i}W=0$ for some $i\in
\{1,2,\ldots ,k\}$.)
\end{itemize}
\end{corollary}

With our final result we will describe the form of all additive, surjective
maps on $H_{n}^{+}(\mathbb{R)}$, $n\geq3$, that preserve the minus partial
order in both directions. Denote by $E_{ij}$ the $n\times n$ matrix with all
entries equal to zero except the $(i,j)$-entry which is equal to one. Let
$E_{k}=E_{11}+E_{22}+\ldots+E_{kk}$. For $A,B\in M_{n}(\mathbb{R)}$ we will
write $A<^{-}B$ when $A\leq^{-}B$ and $A\neq B$. We will denote by $x\otimes
y^{t}$ a rank one linear operator on $\mathbb{R}^{n}$ defined with $(x\otimes
y^{t})z=\left\langle z,y\right\rangle x$ for every $z\in\mathbb{R}^{n}$ (here
$\left\langle z,y\right\rangle =y^{t}z$). Note that every rank-one linear
operator on $\mathbb{R}^{n}$ may be written in this form and that $P\in
P_{n}(\mathbb{F})$ is of rank-one if and only if $P=x\otimes x^{t}$ for some
$x\in\mathbb{R}^{n}$ with $\left\Vert x\right\Vert =1$.

\begin{theorem}
Let $n\geq3$ be an integer. Then $\varphi:H_{n}^{+}(\mathbb{R})\rightarrow
H_{n}^{+}(\mathbb{R})$ is a surjective, additive map that preserves the minus
order $\leq^{-}$ in both directions if and only if there exists an invertible
matrix $S\in M_{n}(\mathbb{R)}$ such that
\[
\varphi(A)=SAS^{t}%
\]
for every $A\in H_{n}^{+}(\mathbb{R})$.
\end{theorem}

\begin{proof}
Let $\varphi:H_{n}^{+}(\mathbb{R})\rightarrow H_{n}^{+}(\mathbb{R})$ be of the
form $\varphi(A)=SAS^{t}$, $A\in H_{n}^{+}(\mathbb{R})$, where $S\in
M_{n}(\mathbb{R)}$ is an invertible matrix. Then $\varphi$ preserves by
(\ref{eq_minus}) the order $\leq^{-}$ in both directions and is clearly
surjective and additive.

Conversely, let $\varphi:H_{n}^{+}(\mathbb{R})\rightarrow H_{n}^{+}%
(\mathbb{R})$ be a surjective, additive map that preserves the order $\leq
^{-}$ in both directions. We will again split the proof into several steps.

1. $\varphi$ \textit{is bijective and }$\varphi(0)=0$\textit{. }Since
$\leq^{-}$ is a partial order and since $\varphi$ preserves this order in both
directions, the proof that $\varphi$ is bijective and that $\varphi(0)=0$ may
be the same as in the first two steps of Theorem \ref{Theorem_main:_Loewner}.

2. $\varphi$ \textit{preserves the rank, i.e. }$\operatorname*{rank}%
(A)=\operatorname*{rank}(\varphi(A))$ \textit{for every} $A\in H_{n}%
^{+}(\mathbb{R})$.
Let $A\in H_{n}^{+}(\mathbb{R})$ with $\operatorname*{rank}%
(A)=k$. By Sylvester's law of inertia there exists an invertible matrix $R\in
M_{n}(\mathbb{R)}$ such that $E_{k}=RAR^{t}$. Clearly (see
(\ref{eq_rank_minus})),
\[
0<^{-}E_{1}<^{-}E_{2}<^{-}\ldots<^{-}E_{n}=I.
\]
Since congruence preserves rank, we have by (\ref{eq_minus})
\[
0<^{-}R^{-1}E_{1}(R^{-1})^{t}<^{-}R^{-1}E_{2}(R^{-1})^{t}<^{-}\ldots
<^{-}R^{-1}E_{k}(R^{-1})^{t}<^{-}\ldots<^{-}R^{-1}E_{n}(R^{-1})^{t}.
\]
From $(R^{-1})^{t}=(R^{t})^{-1}$ and since $\varphi$ preserves the order
$\leq^{-}$ and is injective, we obtain
\begin{equation}
0<^{-}\varphi(R^{-1}E_{1}(R^{t})^{-1})<^{-}\varphi(R^{-1}E_{2}(R^{t}%
)^{-1})<^{-}\ldots<\varphi(A)<^{-}\ldots<^{-}\varphi(R^{-1}(R^{t})^{-1}).
\label{rank_phi_minus}%
\end{equation}

Let $C,D\in M_{n}(\mathbb{R)}$ with $C<^{-}D$ and $\operatorname*{rank}%
(C)=\operatorname*{rank}(D)$. Then by (\ref{eq_rank_minus}),
$\operatorname*{rank}(D-C)=0$ and therefore $D=C$, a contradiction. So, if
$C<^{-}D$, then $\operatorname*{rank}(C)<\operatorname*{rank}(D)$.

Every succeeding matrix in (\ref{rank_phi_minus}) has the rank that is
strictly greater then its predecessor. Since $\operatorname*{rank}%
\varphi(R^{-1}(R^{t})^{-1})\leq n$, it follows that $\operatorname*{rank}%
\varphi(R^{-1}(R^{t})^{-1})=n$ and therefore $\operatorname*{rank}%
(\varphi(A))=k$.

3. \textit{We may without loss of generality assume that} $\varphi(I)=I$. By
the previous step, $\varphi(I)=B$ where $B\in H_{n}^{+}(\mathbb{R})$ is an
invertible (positive definite) matrix. It follows that there exists a positive
definite matrix $\sqrt{B}\in H_{n}^{+}(\mathbb{R})$ such that $\varphi
(I)=\sqrt{B}\sqrt{B}$. Let $\psi:$ $H_{n}^{+}(\mathbb{R})\rightarrow H_{n}%
^{+}(\mathbb{R})$ be defined with
\[
\psi(A)=\left(  \sqrt{B}\right)  ^{-1}\varphi(A)\left(  \sqrt{B}\right)
^{-1}.
\]
Then $\psi$ is a bijective map that preserves the order $\leq^{-}$ in
both directions. Also, $\psi(I)=I$. We will thus from now on assume that
\[
\varphi(I)=I.
\]

4. \textit{There exists a bijective, linear map }$T:\mathbb{R}^{n}%
\rightarrow\mathbb{R}^{n}$ \textit{such that for every }$P\in P_{n}%
(\mathbb{R})$ \textit{the matrix }$\varphi(P)$ \textit{is the orthogonal
projection matrix on }$T(\operatorname{Im}P)$\textit{, i.e. }%
\[
\mathit{\ }\varphi(P)=P_{T(\operatorname{Im}P)}.
\]

Let $P\in M_{n}(\mathbb{R})$ be an idempotent matrix, i.e. $P^{2}=P$. Then
$\mathbb{R}^{n}=\operatorname{Im}P\oplus\operatorname*{Ker}P=\operatorname{Im}%
P\oplus\mathrm{\operatorname{Im}}(I-P)$ and therefore by (\ref{eq_minus}),
$P\leq^{-}I$. Moreover, if $Q\in M_{n}(\mathbb{R)}$ is an idempotent matrix
and if $A\leq^{-}Q$ for $A\in M_{n}(\mathbb{R)}$, then by e.g. \cite[Lemma
2.9]{MarovtRakicDjodjevic}, $A^{2}=A$. Thus for $P\in M_{n}(\mathbb{R})$ we
have
\[
P\leq^{-}I\quad\text{if and only if}\quad P^{2}=P.
\]
Let now $P\in$ $P_{n}(\mathbb{R})$, i.e. $P$ is a symmetric and idempotent
matrix. It follows that $P\leq^{-}I$ and therefore $\varphi(P)\leq^{-}%
\varphi(I)=I$. So, $\varphi(P)$ is an idempotent matrix and by the definition of
the map $\varphi$ also symmetric, i.e. $\varphi(P)\in$ $P_{n}(\mathbb{R})$.
Since $\varphi^{-1}$ has the same properties as $\varphi$, we may conclude
that
\[
P\in P_{n}(\mathbb{R})\quad\text{if and only if}\quad\varphi(P)\in
P_{n}(\mathbb{R}),
\]
i.e. $\varphi$ preserves the set of all orthogonal projection matrices. Recall
that we may identify subspaces of $\mathbb{R}^{n}$ with elements of
$P_{n}(\mathbb{R})$. Let $\mathcal{C}(\mathbb{R}^{n})$ be the lattice of all
subspaces of $\mathbb{R}^{n}$. It follows that the map $\varphi$ induces a
lattice automorphisms, i.e. a bijective map $\tau:\mathcal{C}(\mathbb{R}%
^{n})\rightarrow\mathcal{C}(\mathbb{R}^{n})$ such that
\[
M\subseteq N\quad\text{if and only if}\quad\tau(M)\subseteq\tau(N)
\]
for all $M,N\in\mathcal{C}(\mathbb{R}^{n})$. In \cite[page 246]{Mackey} (see
also \cite[pages 820 and 823]{FillmoreLongstaff} or \cite[page 82]{Pankov})
Mackey proved that for $n\geq3$ every such a map is induced by an invertible
linear operator, i.e. there exists an invertible linear operator $T:$
$\mathbb{R}^{n}\rightarrow\mathbb{R}^{n}$ such that $\tau(M)=T(M)$ for every
$M\in$ $\mathcal{C}(\mathbb{R}^{n})$. For the map $\varphi$ it follows that
\begin{equation}
\varphi(P)=P_{T(\operatorname{Im}P)} \label{eq_projectors}%
\end{equation}
for every $P=P_{\operatorname{Im}P}\in P_{n}(\mathbb{R})$.

5. \textit{We may without loss of generality assume that }$\varphi(P)=P$
\textit{for every }$P\in P_{n}(\mathbb{R})$. Let $x\in\mathbb{R}^{n}$ with
$\left\Vert x\right\Vert =1$. Recall that then $x\otimes x^{t}\in
P_{n}(\mathbb{R})$ is of rank-one. So, by steps 2 and 4 there exists
$a\in\mathbb{R}^{n}$ with $\left\Vert a\right\Vert =1$ such that
\[
\varphi(x\otimes x^{t})=a\otimes a^{t}.
\]
Let $y\in\mathbb{R}^{n}$, $\left\Vert y\right\Vert =1$, and $\left\langle
x,y\right\rangle =0$. We have $\varphi(y\otimes y^{t})=b\otimes b^{t}$ for
some $b\in\mathbb{R}^{n}$, $\left\Vert b\right\Vert =1$. Note that $x\otimes
x^{t}+y\otimes y^{t}\in P_{n}(\mathbb{R})$ and that it is of rank-two. It
follows that $\varphi(x\otimes x^{t}+y\otimes y^{t})$ is a rank-two orthogonal
projection matrix. Since $\varphi$ is additive, we obtain
\[
\varphi(x\otimes x^{t}+y\otimes y^{t})=a\otimes a^{t}+b\otimes b^{t}.
\]
Since this is a rank-two matrix, we may conclude that $a$ and $b$ are linearly
independent vectors. Moreover, from
\[
\left(  a\otimes a^{t}+b\otimes b^{t}\right)  ^{2}=a\otimes a^{t}+b\otimes
b^{t}%
\]
and since $\left\Vert a\right\Vert =\left\Vert b\right\Vert =1$ we get%
\[
\left\langle z,a\right\rangle a+\left\langle z,b\right\rangle b+\left\langle
z,a\right\rangle \left\langle a,b\right\rangle b+\left\langle z,b\right\rangle
\left\langle b,a\right\rangle a=\left\langle z,a\right\rangle a+\left\langle
z,b\right\rangle b
\]
and thus $\left\langle z,a\right\rangle \left\langle a,b\right\rangle
b=-\left\langle z,b\right\rangle \left\langle b,a\right\rangle a$ for every
$z\in\mathbb{R}^{n}$. Let $z=a$ and assume that $\left\langle a,b\right\rangle
\neq0$. Then $b=-\left\langle b,a\right\rangle a$, i.e. $a$ and $b$ are
linearly dependent, a contradiction. It follows that
\[
\left\langle a,b\right\rangle =0.
\]
On the one hand, $\operatorname{Im}\varphi(x\otimes x^{t})=$Lin$\{a\}$ and on
the other hand by (\ref{eq_projectors}) $\operatorname{Im}\varphi(x\otimes
x^{t})=T\left(  \text{Lin}\{x\}\right)  =$Lin$\{Tx\}$. It follows that $a$ and
$Tx$ are linearly dependent, i.e. $a=\mu Tx$ for some $\mu\in\mathbb{R}%
\backslash\{0\}$. Similarly, there exists $\nu\in\mathbb{R}\backslash\{0\}$
such that $b=\nu Ty$. This yields
\[
0=\left\langle \mu Tx,\nu Ty\right\rangle =\mu\nu\left\langle
Tx,Ty\right\rangle =\mu\nu\left\langle T^{t}Tx,y\right\rangle
\]
and therefore $\left\langle T^{t}Tx,y\right\rangle =0$. This equation holds
for every $y\in\mathbb{R}^{n}$ with $\left\Vert y\right\Vert =1$ and
$\left\langle x,y\right\rangle =0$. Since $\left\langle T^{t}Tx,y\right\rangle
=\left\Vert x\right\Vert \left\Vert y\right\Vert \left\langle T^{t}T\frac
{x}{\left\Vert x\right\Vert },\frac{y}{\left\Vert y\right\Vert }\right\rangle
$, we may conclude that for any fixed $x\in\mathbb{R}^{n}$ we have
$\left\langle T^{t}Tx,y\right\rangle =0$ for every $y\in\mathbb{R}^{n}$ with
$\left\langle x,y\right\rangle =0$. So, $T^{t}Tx$ is a scalar multiple of $x$,
i.e. $T^{t}T$ and $I$ are locally linearly dependent. It is known that for
linear operators of rank at least $2$, local linear dependence implies
(global) linear dependence. Note that $T^{t}T\in H_{n}^{+}(\mathbb{R})$.
Therefore,
\[
T^{t}T=\alpha I
\]
for some scalar $\alpha>0$. Let now $Q=\frac{1}{\sqrt{\alpha}}T$. It follows
that $Q^{t}Q=\frac{1}{\alpha}T^{t}T=I$. So, $Q$ is a linear isometry and since
it is also invertible (and thus surjective), it is also coisometry ($QQ^{t}%
=I$). For any $P\in P_{n}(\mathbb{R})$ we thus have $\varphi
(P)=P_{Q(\operatorname{Im}P)}$ where $Q$ is an orthogonal operator,\textit{
}i.e. it may be represented with an (orthogonal) matrix $Q$ where
$QQ^{t}=Q^{t}Q=I$. Therefore, for every $P\in P_{n}(\mathbb{R})$
\[
\operatorname{Im}\varphi(P)=Q(\operatorname{Im}P)=QP(\mathbb{R}^{n}%
)=QPQ^{t}(\mathbb{R}^{n})=\operatorname{Im}QPQ^{t}.
\]
Since clearly $QPQ^{t}\in P_{n}(\mathbb{R})$, we may conclude that
\[
\varphi(P)=QPQ^{t}%
\]
for every $P\in P_{n}(\mathbb{R})$.

Let $\psi:$ $H_{n}^{+}(\mathbb{R}%
)\rightarrow H_{n}^{+}(\mathbb{R})$ be defined with
\[
\psi(A)=Q^{t}\varphi(A)Q.
\]
Then $\psi$ still preserves the order $\leq^{-}$ and is bijective. Moreover
$\psi(P)=P$ for every $P\in P_{n}(\mathbb{R})$. We will thus from on assume
that
\[
\varphi(P)=P
\]
for every $P\in P_{n}(\mathbb{R})$.

6. $\varphi(\lambda P)=\lambda\varphi(P)$ \textit{for every} $P\in
P_{n}(\mathbb{R})$\textit{ of rank-one and every }$\lambda\in\left[
0,\infty\right)  $\textit{. }Let $P\in P_{n}(\mathbb{R})$ be of rank-one and
let $\lambda>0$. Since $\varphi$ preserves the rank, there exists by the
spectral theorem $Q\in P_{n}(\mathbb{R})$ of rank-one and $\mu>0$ such that
\[
\varphi(\lambda P)=\mu Q.
\]
Suppose $P\neq Q$. Then $P+\alpha Q$ is of rank-two for every scalar
$\alpha>0$. Since $\varphi$ is additive, we obtain
\begin{align*}
\varphi(P+\lambda P)  &  =\varphi(P)+\varphi(\lambda P)\\
&  =P+\mu Q.
\end{align*}
So, on the one hand $\varphi(P+\lambda P)$ is of rank-two but on the other
hand $(1+\lambda)P$ is of rank-one and therefore, since $\varphi$ preserves the rank, $\varphi(P+\lambda
P)=\varphi((1+\lambda)P)$ is of rank-one, a contradiction. It follows that $P=Q$ and
therefore there exists a function $f_{P}:\left[  0,\infty\right)  \rightarrow$
$\left[  0,\infty\right)  $ such that
\[
\varphi(\lambda P)=f_{P}(\lambda)P.
\]
Since $\varphi(P)=P$ and $\varphi(0)=0$, we have $f_{P}(1)=1$ and $f_{P}%
(0)=0$. From
\[
f_{P}(\lambda+\mu)P=\varphi((\lambda+\mu)P)=\varphi(\lambda P)+\varphi(\mu
P)=f_{P}(\lambda)P+f_{P}(\mu)P
\]
we may conclude that $f_{p}$ is additive, i.e. $f_{P}(\lambda+\mu)=$
$f_{P}(\lambda)+f_{P}(\mu)$ for every $\lambda,\mu\in\left[  0,\infty\right)
$. Let $r$ be an arbitrary (but fixed) positive integer. Since $f_{p}$ is additive, it follows that
\[
1=f_{P}(1)=f_{P}\left(  r\frac{1}{r}\right)  =rf_{P}\left(  \frac{1}%
{r}\right)
\]
and thus $f_{P}\left(  \frac{1}{r}\right)  =\frac{1}{r}$. Let now $\frac{q}%
{r}$ be any (but fixed) nonnegative rational number (here $q$ and $r$ are nonnegative and positive integers, respectively). Then, again by the additivity of $f_{p}$,
\begin{equation}
f_{P}\left(  \frac{q}{r}\right)  =qf_{P}\left(  \frac{1}{r}\right)  =\frac
{q}{r}. \label{eq_rational}%
\end{equation}
Note that $f_{p}$ is monotone increasing. Namely, for $\lambda,\mu\in\left[
0,\infty\right)  $ with $\lambda\leq\mu$ we have $\mu=\lambda+\nu$ for some
$\nu\geq0$. Thus, $f_{P}(\lambda)\leq f_{P}(\lambda)+f_{P}(\nu)=f_{P}(\mu)$.

Let $\lambda\in\left(  0,\infty\right)  $ be arbitrary. Then
$\lambda$ is a limit of a monotone increasing sequence $\left\{
s_{i}\right\}  $ of nonnegative rational numbers and a limit of a monotone decreasing
sequence $\left\{  z_{i}\right\}  $ of positive rational numbers. Since for every
$i\in\mathbb{N}$, we have by (\ref{eq_rational}), $f_{P}(s_{i})=s_{i}$ and
$f_{p}(z_{i})=z_{i}$, it follows by the monotonicity of $f_{P}$ that
$$f_{P}(\lambda)=\lambda$$ for every $\lambda\in\left(  0,\infty\right)  $. Recall that $f_{P}(0)=0$. It
follows that
\begin{equation}
\varphi(\lambda P)=\lambda\varphi(P) \label{eq_proj_lambda}%
\end{equation}
for every rank-one $P\in P_{n}(\mathbb{R})$ and every $\lambda\in\left[
0,\infty\right)  $.

We are now in position to conclude the proof of the theorem. Let $A\in
H_{n}^{+}(\mathbb{R})$ be arbitrary. By the spectral theorem there exist
pairwise orthogonal rank-one (idempotent and symmetric) matrices $P_{1}%
,P_{2},\ldots,P_{k}\in P_{n}(\mathbb{R})$ and $\lambda_{1},\lambda_{2}%
,\ldots,\lambda_{k}\in$ $\left[  0,\infty\right)  $ such that
\[
A=\lambda_{1}P_{1}+\lambda_{2}P_{2}+\ldots+\lambda_{k}P_{k}.
\]
By (\ref{eq_proj_lambda}) and since $\varphi$ is additive, we may conclude
that
\[
\varphi(A)=A
\]
for every $A\in H_{n}^{+}(\mathbb{R})$. To sum up, taking into account our
assumptions, a surjective, additive map $\varphi:H_{n}^{+}(\mathbb{R}%
)\rightarrow H_{n}^{+}(\mathbb{R})$, $n\geq3$, that preserves the minus order
$\leq^{-}$ in both directions is of the following form:
\[
\varphi(A)=SAS^{t}%
\]
for every $A\in H_{n}^{+}(\mathbb{R})$ where $S\in M_{n}(\mathbb{R})$ is an
invertible matrix.
\end{proof}

\begin{remark}
We believe that the same result holds also without the additivity assumption and it would be interesting to find a proof of this conjecture. Also, we expect that a surjective map $\varphi:H_{2}^{+}(\mathbb{R})\rightarrow
H_{2}^{+}(\mathbb{R})$ that preserves the minus order in both directions has the form $\varphi(A)=SAS^{t}$ for every $A\in H_{2}^{+}(\mathbb{R})$ where $S\in M_{2}(\mathbb{R)}$ is an invertible matrix.
\end{remark}

\section{Concluding remarks}

Many other partial orders may be defined on $M_{n}(\mathbb{F})$ where $\mathbb{F=R}$ or $\mathbb{F=C}$. The star
partial order $\leq^{\ast}$ is defined in the following way (see \cite{Drazin}): For $A,B\in M_{n}(\mathbb{F})$ we write
\[
A\leq^{\ast}B\quad\text{when}\quad A^{\ast}A=A^{\ast}B\text{ and }AA^{\ast
}=BA^{\ast}.
\]
It is known (see e.g. \cite{MitraKnjiga}) that $A\leq^{\ast}B$ implies
$A\leq^{-}B$. Two partial orders that are "related" to the minus and the star
partial orders are the left-star and the-right star partial orders
\cite{BaksalaryMitra}. For $A,B\in M_{n}(\mathbb{F})$ we say that $A$ is below
$B$ with respect to the left-star partial order and write
\[
A\ast\!\!\leq B\quad\text{when}\quad A^{\ast}A=A^{\ast}B\text{ and
}\mathrm{\operatorname{Im}}A\subseteq\operatorname{Im}B.
\]
Similarly, we define the right-star partial order: For $A,B\in M_{n}%
(\mathbb{F})$ we write
\[
A{\leq\!\!\!\ast}\,B\quad\text{when}\quad AA^{\ast}=AB^{\ast}\text{ and
}\mathrm{\operatorname{Im}}A^{\ast}\subseteq\operatorname{Im}B^{\ast}.
\]
It is known (see \cite{MitraKnjiga}) that for $A,B\in M_{n}(\mathbb{F})$,
$A\leq^{\ast}B$ implies both $A\ast\!\!\leq B$ and $A{\leq\!\!\!\ast}\,B$ and
each $A\ast\!\!\leq B$ and $A{\leq\!\!\!\ast}\,B$ implies $A\leq^{-}B$. The
converse implications do not hold in general. Note that the left-star partial order has applications in the theory of linear models (see \cite[Theorem 15.3.7, Corollary
15.3.8]{MitraKnjiga}).

Let $A,B\in H_{n}^{+}%
(\mathbb{F})$. Since then $A^{\ast}A=A^{\ast}B$ if and only if $(A^{\ast}A)^{\ast}=(A^{\ast}B)^{\ast}$ if and only if $A^{2}=BA$ which is equivalent to $AA^{\ast}=BA^{\ast}$, we may conclude that the star, the left-star, and the right-star partial orders are the same partial order on  $H_{n}^{+}(\mathbb{F})$. Maps on $M_{n}(\mathbb{F})$ preserving these orders have already been studied
(see \cite{Guterman, Legisa3}). It would be interesting to describe
(surjective) maps that preserve the star order (in both directions) on the set $H_{n}^{+}(\mathbb{F})$ of all real or complex positive semidefinite matrices.

\end{document}